\documentclass[12pt]{amsart}
\usepackage{amsmath,amssymb,amsbsy,amsfonts,latexsym,amsopn,amstext,
                                               amsxtra,euscript,amscd,bm}
\usepackage{url}

\usepackage{mathrsfs}

\usepackage{color}
\usepackage[colorlinks,linkcolor=blue,anchorcolor=blue,citecolor=blue,backref=page]{hyperref}

\usepackage{graphics,epsfig}
\usepackage{graphicx}
\usepackage{float}
\usepackage{epstopdf}
\hypersetup{breaklinks=true}
\usepackage{verbatim}
\usepackage{graphicx}
\usepackage{caption}
\usepackage{subcaption}

\usepackage{soul} 
\usepackage{ulem}

\usepackage{float}
\usepackage{epstopdf}
\hypersetup{breaklinks=true}

\usepackage[np]{numprint}
\npdecimalsign{\ensuremath{.}}

\usepackage{bibentry}

\usepackage[english]{babel}
\usepackage{mathtools}
\usepackage{todonotes}
\usepackage{url}
\usepackage[norefs,nocites]{refcheck}

\newtheorem{theorem}{Theorem}
\newtheorem*{thm*}{Theorem}
\newtheorem{lemma}[theorem]{Lemma}

\newtheorem{prop}[theorem]{Proposition}

\numberwithin{equation}{section}
\numberwithin{theorem}{section}
\numberwithin{table}{section}

\newcommand{\sq}{\operatorname{sq}}

\newcommand{\Z}{{\mathbb Z}} 

\newcommand{\N}{{\mathbb N}}

\newcommand{\lcm}{\operatorname{lcm}}

   \def \e{\varepsilon}   \def \l{\lambda}  \def \t{\vartheta}

\newfont{\teneufm}{eufm10}
\newfont{\seveneufm}{eufm7}
\newfont{\fiveeufm}{eufm5}
\newfam\eufmfam
     \textfont\eufmfam=\teneufm
\scriptfont\eufmfam=\seveneufm
     \scriptscriptfont\eufmfam=\fiveeufm
\def\fU{{\mathfrak U}}

\def \balpha{\bm{\alpha}}
\def \bbeta{\bm{\beta}}

\def\eqref#1{(\ref{#1})}

\def\le{\leqslant}
\def\leq{\leqslant}
\def\ge{\geqslant}
\def\leq{\leqslant}

\def\cN{{\mathcal N}}

\def\cR{{\mathcal R}}
\def\cS{{\mathcal S}}
\def\cT{{\mathcal T}}

\def\cX{{\mathcal X}}

\def\cZ{{\mathcal Z}}

\def\e{{\mathbf{\,e}}}

\def\lcm{{\mathrm{lcm}}}
\def\ls#1#2{\(\frac{#1}{#2}\)}

 \def\0{{\mathbf{0}}}

\def\({\left(}
\def\){\right)}
\def\l|{\left|}
\def\r|{\right|}

\def\mand{\qquad \mbox{and} \qquad}

\def \fP {\mathfrak P}

\newif\ifcomment

\usepackage{graphicx}
\begin{document}

\title[A large sieve inequality over short intervals]{A large sieve inequality for sums of Legendre symbols over short intervals}

\author[M. Munsch] {Marc Munsch}
\address{
Université Jean Monnet, Centrale Lyon, INSA Lyon, Université Claude Bernard Lyon 1, CNRS, ICJ UMR5208, 42023 Saint-Étienne, France}
\email{marc.munsch@univ-st-etienne.fr}

\author[I. E. Shparlinski] {Igor E. Shparlinski}
\address{School of Mathematics and Statistics, University of New South Wales, Sydney NSW 2052, Australia}
\email{igor.shparlinski@unsw.edu.au}

\author[Y.-C. Sun] {Yu-Chen Sun}
\address{
Department of Mathematics,
University of Bristol, Woodland Rd, Bristol BS8 1UG, UK}
\email{yuchensun93@163.com}

\author[Y. Xiao] {Yixiu Xiao}
\address{School of Mathematical Sciences, Shanghai Jiao Tong University, 800 Dongchuan RD, 200240 Shanghai, China}
\email{yixiuxiao98@gmail.com}

\begin{abstract}
Using the Burgess bound and the Selberg sieve, we obtain an upper bound for
the second moment of sums of Legendre symbols over intervals $[u+1,u+h]$,
with the modulus ranging over primes $p\in[Q,2Q]$. The bound is nontrivial
and yields a power saving in $h$, uniformly for $u\le Q$, provided that
$h\ge\psi(Q)$, where $\psi(Q)\to\infty$ as $Q\to\infty$. This may be
viewed as a short-interval analogue of a result of D.~R.~Heath-Brown (1995)
on moments of quadratic character sums over the initial interval $[1,h]$.
In particular, it implies that, for any prescribed interval of this length,
the quadratic residues and non-residues are asymptotically equidistributed
for almost all primes $p$. We also establish estimates for higher moments
conditionally on the Generalised Riemann Hypothesis. These bounds rely on a 
sharp uniform estimate for the number of tuples of
integers in a shifted interval whose product is a square. 
\end{abstract}

\maketitle

\tableofcontents

\section{Introduction} 
\subsection{Set-up and background on character sums}
The distribution of quadratic residues and non-residues modulo a prime $p$
is a central problem in number theory. It is particularly difficult to
understand this distribution in short intervals of the form $[u+1,u+h]$
when $h$ is much smaller than $u$. Even the Generalised Riemann Hypothesis
(GRH) does not appear to yield useful information in this setting. We
therefore study the problem on average over $p$, using classical techniques
from analytic number theory. Equivalently, we seek bounds for short sums of
Legendre symbols modulo $p$, averaged over $p$.

For a prime $p\ge 3$ and real numbers $u\ge 0$ and $h\ge 1$, we consider the
character sum
\[
S_p(u,h)=\sum_{u<n\le u+h}\ls{n}{p}.
\]  
The celebrated Burgess bound~\cite{Burg} gives
\[
S_p(u,h)=o(h),
\]
provided that $h\ge p^{1/4+\varepsilon}$ for some fixed
$\varepsilon>0$; see~\cite{dlBMT,KSY} for recent refinements. For initial
intervals, corresponding to $u=0$, it is known~\cite{BGH-BS} that if
$h\ge p^{1/(4\sqrt e)+\varepsilon}$, then $|S_p(0,h)|\le\eta h$ for some
$\eta=\eta(\varepsilon)<1$.

Moreover, Konyagin and Shparlinski~\cite{KonShp} proved that, for any
function $\psi(Q)\to\infty$ as $Q\to\infty$ and any prescribed $u\le Q$,
the inequality $S_p(u,h)<h$ holds for all but $o(Q/\log Q)$ primes
$p\in[Q,2Q]$, provided that $h\ge\psi(Q)$. Thus, for almost all such
primes, the interval $(u,u+h]$ contains a quadratic non-residue. For $u=0$,
this also follows from a stronger result of Erd{\H o}s~\cite{Erd}; see also
Linnik's celebrated theorem~\cite{Lin}.

For initial intervals, it is also well known that the GRH yields essentially
square-root cancellation, $|S_p(0,h)|\le h^{1/2+o(1)}$, for considerably
shorter intervals. Although a direct reference does not appear to be readily
available, this estimate can be deduced from~\cite[Theorem~2]{GrSo}; see
also Lemma~\ref{lem:GRH}. It is unclear, however, whether this approach
yields useful estimates for short intervals away from the origin.

Our aim is to obtain nontrivial bounds for moments of these sums over primes
$p\in[Q,2Q]$, where $Q$ is a positive real parameter. More generally, we
consider the weighted sums
\[
S_p(\balpha,u,h)=\sum_{u<n\le u+h}\alpha_n\ls{n}{p},
\]  
where $\balpha=(\alpha_n)_{u<n\le u+h}$ is a sequence of complex weights
supported on the half-open interval $(u,u+h]$ and satisfying
\[
|\alpha_n|\le1,\qquad u<n\le u+h.
\]

For $s\ge1$, define
\[
M_s(Q;\balpha,u,h)=\sum_{\substack{p\in[Q,2Q]\\p~\text{prime}}}
 |S_p(\balpha,u,h)|^s.
\]

When $u=0$, Heath-Brown's result~\cite[Theorem~1]{H-B1} immediately gives
\begin{equation}
\label{eq: HB-Bound}
M_2(Q;\balpha,0,h)\le(hQ)^{o(1)}(Q+h)h,
\end{equation}
as $h,Q\to\infty$. The standard bound for the divisor function,
see~\cite[Equation~(1.81)]{IwKow}, extends this estimate to higher moments:
\[
M_{2s}(Q;\balpha,0,h)\le(hQ)^{o(1)}(Q+h^s)h^s,
\]
for every fixed positive integer $s$.

For general $u$, Heath-Brown~\cite[Lemma~5]{H-B2} also obtained a
first-moment bound for the absolute values of certain unweighted variants
of $S_p(\balpha,u,h)$. This bound is nontrivial when $h$ is relatively
large compared with $Q$, namely when $h\ge Q^{2/5+\varepsilon}$.

Lamzouri~\cite{Lam} and, more recently, Harper~\cite{Harp} studied the
distribution of $S_p(\balpha,u,h)$, and of analogous sums for other
nonprincipal multiplicative characters, with $p$ fixed and $u$ varying.
Tang and Zhang~\cite{TaZh} studied averages over $p$ of Legendre-symbol
sums over initial intervals. Recent work of the first author and
Toma~\cite{MuTo} contains both new results and a comprehensive account of
earlier work on moments of quadratic-character sums over initial intervals.
None of these results, however, addresses the problem considered here.

 \subsection{Main results}
 
Throughout the paper, the notation
\[
X=O(Y),\qquad X\ll Y,\qquad Y\gg X,
\]
means that $|X|\le cY$ for some positive constant $c$. This constant
may depend on the fixed positive integers $\ell$, $r$ and $s$ when such
dependence is clear from context; otherwise it is absolute. For $X\ge1$,
we write $X^{o(1)}$ for a positive quantity $f(X)$ satisfying
\[
\lim_{X\to\infty}\frac{\log f(X)}{\log X}=0.
\]

Combining and extending ideas from~\cite{KonShp,OstShp}, we use the
Selberg sieve to estimate the even moments $M_{2s}(Q;\balpha,u,h)$.

We first note that when $Q\ge h\ge u$, Heath-Brown's
bound~\eqref{eq: HB-Bound}, applied to the extended sequence below, gives
\begin{equation}
\label{eq: HB-Bound Shift}
M_2(Q;\balpha,u,h)=M_2(Q;\widetilde\balpha,0,u+h)
 \le Q^{1+o(1)}h,
\end{equation}
where $\widetilde\alpha_n=0$ for $n\le u$ and
$\widetilde\alpha_n=\alpha_n$ otherwise.

When $h$ is a fixed positive power of $Q$, estimate~\eqref{eq: HB-Bound Shift}
exhibits
square-root cancellation on average. For smaller $h$, however, it is
trivial, even when $u=0$.

We henceforth focus on the range $Q\ge u\ge h$. As usual, $\pi(Q)$
denotes the number of primes not exceeding $Q$.

We first give a bound on the second moment $M_{2}(Q; \balpha, u,h)$.

 \begin{theorem}
 \label{thm: Bound M2Quh}  
For all integers $2\le h<u\le Q$, we have
\[
M_2(Q;\balpha,u,h)\ll h\pi(Q)\log h+h^2Q^{7/8+o(1)}.
\]
\end{theorem}  

Theorem~\ref{thm: Bound M2Quh} gives essentially square-root cancellation
for almost all primes $p\in[Q,2Q]$ whenever
\[
\psi(Q)\le h\le Q^{1/8-\varepsilon},
\]
where $\psi(Q)\to\infty$ and $\varepsilon>0$ is fixed. For unweighted
sums, this implies that, for each prescribed interval in this range,
the quadratic residues and non-residues are asymptotically equidistributed
for almost all primes $p\in[Q,2Q]$.

Inspection of the proof, particularly Section~\ref{sec: Bound U}, shows
that the factor $Q^{7/8+o(1)}$ in the second term of
Theorem~\ref{thm: Bound M2Quh} may be replaced by
$Q^{1/2+o(1)}u^{3/8}$. This refinement is not needed here. The essential
feature of the theorem is the factor $\pi(Q)$, rather than the more
immediate factor $Q$, in the first term; this is what makes the estimate
nontrivial for very short sums. See Section~\ref{sec:comm}.

Higher moments are more sensitive to exceptionally large character sums.
It is therefore natural to study $M_{2s}(Q;\balpha,u,h)$ in pursuit of
nontrivial pointwise bounds for every $p$. Our higher-moment estimates are
conditional on the GRH.

 \begin{theorem}
\label{thm: Bound MsQuh}
Let $s\ge1$ be a fixed integer. Under the GRH, for all integers
$2\le h<u\le Q$, we have 
\[
M_{2s}(Q;\balpha,u,h)
\ll h^s(\log h)^{s(2s-1)}\pi(Q)+h^{2s}Q^{1/2+o(1)}.
\]
\end{theorem}  

For $s=1$, Theorem~\ref{thm: Bound MsQuh} strengthens
Theorem~\ref{thm: Bound M2Quh} under the GRH.

Both theorems rely on an estimate concerning the ``anatomy of integers''
that may be of independent interest.

To state this estimate, we use $\square$ to denote an unspecified perfect
square and
define the set 
\[
\cR_s(h,u)=\{(n_1,\ldots,n_s)\in\mathbb Z^s\cap(u,u+h]^s:
n_1\cdots n_s=\square\}.
\] 

Let $R_s(h,u)=\#\cR_s(h,u)$.  {We prove the following optimal estimate for the number of squares from short intervals.

\begin{theorem}\label{thm:short_bound}
Let $s\ge2$ be fixed. Uniformly for $u$ and $h \ge 2$, we have
\[
R_s(h,u)\ll h^{s/2}(\log h)^{s(s-1)/2}.
\]\ 
\end{theorem}

For initial intervals (that is, for $u=0$), the asymptotic formula 
\begin{equation}
\label{eq: Asymp}
R_s(h,0) \sim C_s h^{s/2}(\log h)^{s(s-1)/2}
\end{equation}
is known by~\cite[Corollary~2.4]{dlBMTKS}, see also~\cite[Lemma~15]{MuTo}.

\section{Preliminaries}

\subsection{Background on the Selberg sieve}
 
We recall the properties of the Selberg weights needed below. Our
formulation follows~\cite{OstShp}, which in turn draws
on~\cite[Section~3.2]{MonVau2}; see also~\cite{Harp}.
 
Let $z$ be a real parameter satisfying $z^3\le Q\le z^{O(1)}$, and define
\begin{equation}
\label{eq: PrimeProd}
\fP=\prod_{p\le z}p.
\end{equation}

There are real coefficients $\Lambda_d$, chosen as
in~\cite[Section~3.2]{MonVau2}, for which the associated Selberg weights
\[
\lambda_n^+=\sum_{\substack{r,s=1\\\lcm[r,s]=n}}^\infty
 \Lambda_r\Lambda_s,
\]
have the following properties:
\begin{itemize}
\item We have
\begin{equation}
\label{eq:lambda_n = 0}
\lambda_n^+=0,\qquad n\ge z^2.
\end{equation}
\item Equations~(3.17) and~(3.22) of~\cite{MonVau2}, together
with~\cite[Exercise~2.1.17]{MonVau2}, give
\begin{equation}\label{eq: sum|lambda_n|}
\sum_{n=1}^\infty|\lambda_n^+|\ll\frac{z^2}{(\log z)^2}.
\end{equation}
\item The definition of $\lambda_n^+$ and~\cite[Equations~(3.12)
and~(3.13)]{MonVau2} give
\begin{equation}\label{eq: sum lambda_e e|q}
\sum_{e\mid q}\lambda_e^+\ge
\begin{cases}
1,&\gcd(q,\fP)=1,\\
0,&\text{otherwise.}
\end{cases}
\end{equation}
where $\fP$ is defined by~\eqref{eq: PrimeProd};
\item The argument in the proof of~\cite[Theorem~3.2]{MonVau2}, with
the above choice of $\fP$ and the real coefficients $\Lambda_d$, shows
that, for every $Z$ satisfying $z^3\le Z\le z^{O(1)}$,
 \begin{equation}
\label{eq: sum lambda_n dyadic}
\begin{split}
\sum_{q\in[Z,2Z]}\sum_{e\mid q}\lambda_e^+
&=Z\sum_{e\le z^2}\frac{\lambda_e^+}{e}+O(z^2)\\
&=Z\sum_{e\le z^2}\frac{1}{e}
  \sum_{\substack{r,s\le z\\\lcm[r,s]=e}}\Lambda_r\Lambda_s+O(z^2)
 \ll\frac{Z}{\log Z},
 \end{split} 
\end{equation}
for the choice of the coefficients $\Lambda_d$ specified
in~\cite[Equations~(3.12) and~(3.15)]{MonVau2}.
\end{itemize}

\subsection{Character sums}
We first recall the following special case of the classical Burgess bound
for sums of Jacobi symbols with arbitrary odd moduli. It follows
from~\cite[Theorem~12.6]{IwKow} by taking $r=2$.

\begin{lemma}
\label{lem:Burg}
For every real $A$, every real $V\ge1$, and every odd nonsquare integer
$m\ge1$, we have
 \[
\left|\sum_{A\le v\le A+V}\ls{v}{m}\right|
 \le V^{1/2}m^{3/16+o(1)},\qquad m\to\infty.
\]
\end{lemma}

A stronger estimate is available under the GRH; see
~\cite[Section~1]{MonVau1}. It may also be deduced
from~\cite[Theorem~2]{GrSo}.

\begin{lemma}
\label{lem:GRH}
Assume the GRH. Then, for every real $T\ge1$ and every odd nonsquare
integer $m\ge1$, we have
\[
\left|\sum_{1\le t\le T}\ls{t}{m}\right|
 \le T^{1/2}m^{o(1)},\qquad m\to\infty.
\]
\end{lemma}

\subsection{Generalised  Cauchy--Schwarz inequality} 

We need the following Cauchy--Schwarz inequality for pair-indexed
variables. 
It is a special case of the Finner inequality~\cite{Finner}, we include an elementary proof for the sake of completeness.

\begin{lemma}
\label{lem:pair-CS}  
Let $m\geq2$.  For $1\leq i<j\leq m$, let $\cX_{ij}$ be a finite set,
and let $x_{ij}\in\cX_{ij}$ be a variable.  Put
$\cX_{ji}=\cX_{ij}$ and $x_{ji}=x_{ij}$. For $1\leq i\leq m$, define the Cartesian product and vector
\[
\cX_i=\prod_{\substack{1\leq j\leq m\\j\ne i}}\cX_{ij},
\qquad
\mathbf{x}_i=(x_{ij})_{\substack{1\leq j\leq m\\j\ne i}}\in\cX_i,
\]  
and let $F_i:\cX_i\longrightarrow[0,\infty)$.  Then
\[
\sum_{\substack{x_{ij}\in\cX_{ij}\\1\leq i<j\leq m}}
 \prod_{r=1}^m F_r(\mathbf{x}_r)
\leq
\prod_{r=1}^m
\left(
 \sum_{\mathbf{x}_r\in\cX_r}
 F_r(\mathbf{x}_r)^2
\right)^{1/2}.
\]
\end{lemma}

\begin{proof}
We argue by induction on $m$.  For $m=2$, this is the classical Cauchy--Schwarz inequality.

Assume $m\geq3$ and that the result is known with $m$ replaced by
$m-1$.  For
$1\leq i<j\leq m-1$, write
\[
 z_{ij}=x_{ij},\qquad z_{ji}=z_{ij},
\]
and, for $1\leq i\leq m-1$, put
\[
\cZ_i=\prod_{\substack{1\leq j\leq m-1\\j\ne i}}\cX_{ij},
\qquad
\mathbf{z}_i=(z_{ij})_{\substack{1\leq j\leq m-1\\j\ne i}}
\in\cZ_i,
\qquad
y_i=x_{im}\in\cX_{im}.
\]
Thus, $\mathbf{z}_i$ collects the variables $z_{ij}$ with
$1\leq j\leq m-1$ and $j\ne i$. 
We have 
\[
\cX_i=\cZ_i\times\cX_{im}\quad (i<m),  \mand 
\cX_m=\prod_{r=1}^{m-1}\cX_{rm}. 
\]
Next, we set 
\[
\mathbf{y}=(y_1,\ldots,y_{m-1})
\in\cX_m,
\qquad
\mathbf{x}_m=\mathbf{y},
\]
and
\[
 G(\mathbf{y})
 =
 \sum_{\substack{z_{ij}\in\cX_{ij}\\1\leq i<j\leq m-1}}
 \prod_{r=1}^{m-1}F_r(\mathbf{z}_r,y_r).
\]
The original sum is
\[
\sum_{\substack{x_{ij}\in\cX_{ij}\\1\leq i<j\leq m}}
 \prod_{r=1}^m F_r(\mathbf{x}_r) = \sum_{\mathbf{y}\in \cX_{m}}
 F_m(\mathbf{y})G(\mathbf{y}).
\]
The classical Cauchy--Schwarz inequality in the variables $y_1,\ldots,y_{m-1}$ gives
\begin{equation}
 \label{eq:first-CS}
\sum_{\mathbf{y}\in \cX_{m}}
 F_m(\mathbf{y})G(\mathbf{y})
 \leq \|F_m\|_2\,\|G\|_2
\end{equation}
where 
\[
\|F_m\|_2 = \( \sum_{\mathbf{y}\in \cX_{m}}
 F_m(\mathbf{y})^2\)^{1/2} \mand \|G\|_2 =  \( \sum_{\mathbf{y}\in \cX_{m}}
G(\mathbf{y})^2\)^{1/2}.
\] 
To estimate $\|G\|_2$, for $1\leq i<j\leq m-1$ , we put
\[
\widetilde{\cX}_{ij}=\cX_{ij}\times\cX_{ij},
\qquad
\widetilde{\cX}_{ji}=\widetilde{\cX}_{ij}.
\]
Write
\[
w_{ij}=(z_{ij},z_{ij}^*)\in\widetilde{\cX}_{ij},
\qquad
w_{ji}=w_{ij}.
\]
Thus $z_{ji}^*=z_{ij}^*$.  For $1\leq r\leq m-1$, put
\[
\widetilde{\cX}_r
 =\prod_{\substack{1\leq j\leq m-1\\j\ne r}}
   \widetilde{\cX}_{rj},
\qquad
\mathbf{w}_r
 =(w_{rj})_{\substack{1\leq j\leq m-1\\j\ne r}}
 \in\widetilde{\cX}_r,
\]
and set
\[
\mathbf{z}_r^*
 =(z_{rj}^*)_{\substack{1\leq j\leq m-1\\j\ne r}}\in\cZ_r.
\]
Under the natural interpretations 
\[
\widetilde{\cX}_r =  \cZ_r\times\cZ_r,
\qquad
\mathbf{w}_r = (\mathbf{z}_r,\mathbf{z}_r^*),
\]
define $H_r: \widetilde{\cX}_r\longrightarrow[0,\infty)$ by
\[
H_r(\mathbf{w}_r)
 =
\sum_{t\in\cX_{rm}}
F_r(\mathbf{z}_r,t)F_r(\mathbf{z}_r^*,t).
\]
Expanding the square of $\|G\|_2$ now gives
\[
 \|G\|_2^2
 =
 \sum_{\substack{w_{ij}\in\widetilde{\cX}_{ij}\\
                  1\leq i<j\leq m-1}}
 \prod_{r=1}^{m-1}H_r(\mathbf{w}_r).
\]
The induction hypothesis, applied to the sets
$\widetilde{\cX}_{ij}$ and the functions $H_r$, yields
\[
 \|G\|_2^2
 \leq
 \prod_{r=1}^{m-1}
 \left(
 \sum_{\mathbf{w}_r\in\widetilde{\cX}_r}
 H_r(\mathbf{w}_r)^2
 \right)^{1/2}.
\]
For each $1\leq r\leq m-1$, another application of the Cauchy--Schwarz inequality 
gives
\[
 H_r(\mathbf{w}_r)^2
 \leq
 \left(\sum_{t\in\cX_{rm}}F_r(\mathbf{z}_r,t)^2\right)
 \left(\sum_{t\in\cX_{rm}}F_r(\mathbf{z}_r^*,t)^2\right).
\]
After summing over $\mathbf{w}_r\in\widetilde{\cX}_r$, we obtain
\[
 \left(
 \sum_{\mathbf{w}_r\in\widetilde{\cX}_r}
 H_r(\mathbf{w}_r)^2
 \right)^{1/2}
 \leq
 \sum_{\substack{\mathbf{z}_r\in\cZ_r\\t\in\cX_{rm}}}
 F_r(\mathbf{z}_r,t)^2
 =\|F_r\|_2^2.
\]
Thus $\|G\|_2\leq\prod_{r=1}^{m-1}\|F_r\|_2$.  Substitution into
\eqref{eq:first-CS} completes the induction.
\end{proof}

\section{Square products from shifted intervals}

\subsection{Reduction to pairwise distinct integers}
We first separate the contribution of tuples containing repeated
coordinates. Put
\[
\cR_s^*(h,u)
 =\{(n_1,\ldots,n_s)\in\cR_s(h,u):
      n_i\ne n_j,\ \text{for }i\ne j\},
\]
and let $R_s^*(h,u)=\#\cR_s^*(h,u)$.

\begin{lemma}\label{lem_reduction}
For every fixed integer $s\ge 1$, uniformly for  $u$ and $h\ge 2$, we have
\[
R_s(h,u)
 \ll h^{s/2}
 +\sum_{j=0}^{\lfloor (s-1)/2\rfloor}h^jR_{s-2j}^*(h,u).
\] 
\end{lemma}

\begin{proof}
If a tuple counted by $R_s(h,u)$ is not pairwise distinct, then two
of its coordinates, say $n_i$ and $n_j$, are equal. There are
$O(1)$ choices for the pair $(i,j)$ and $O(h)$ choices for their
common value. Removing these two coordinates preserves the condition
that the product is a square. Hence
\[
R_s(h,u)\le R_s^*(h,u)+O\(hR_{s-2}(h,u)\).
\]
Iterating this inequality gives the result, with the term $h^{s/2}$
accounting for the case in which all coordinates are removed in
pairs.
\end{proof}

Next, we estimate the number of pairwise distinct solutions. Together with
Lemma~\ref{lem_reduction}, the following bound yields
Theorem~\ref{thm:short_bound},  after noticing that in the case $h\ge u$ 
the result is immediate by~\eqref{eq: Asymp}.

\begin{prop}\label{bound_nontrivial}
Let $s\ge1$ be fixed. Uniformly for $2\le h<u$, we have
\[
R_s^*(h,u)\ll h^{s/2}(\log h)^{s(s-1)/2}.
\]
\end{prop}

\subsection{Small values of $h$}

We first handle the small values of $h$ in 
Proposition~\ref{bound_nontrivial}.  

\begin{lemma}
\label{lem:large-shift-distinct}
For every fixed integer $s\geq1$, there exists a constant $C(s)>0$ with
the following property.  If $2\leq h<u$ and $u>C(s) h^{2s}$, then, for
every $1\leq r\leq s$,
\[
R_r^*(h,u)=0,\quad \text{if $r$ is even,} \mand 
R_r^*(h,u)\leq r!,\quad\text{if $r$ is odd}.
\]
\end{lemma}

\begin{proof}
We first prove an elementary claim for an even number of variables.
For every fixed even integer $r=2k\geq2$, there is a constant $c(r)>0$
such that, if $u>c(r)h^r$, there are no $r$ pairwise distinct integers in
$(u,u+h]$ with a square product.

Indeed, let $x=\lfloor u\rfloor$ and, after reordering, write the
integers as
\[
 x+t_1,\ldots,x+t_{2k}, \qquad 1\leq t_1<\cdots<t_{2k}\leq h+1.
\]
Set
\[
 F(X)=\prod_{i=1}^{2k}(X+t_i)  =X^{2k}+c_1X^{2k-1}+\cdots+c_{2k}.
\]
Since $c_j$ is an elementary symmetric polynomial of degree $j$ in
the variables $t_i$, we have $|c_j| \ll h^j$ where, until the end of the proof, the implied constant may depend on $r$ (or equivalently, on $k$). Expanding $\sqrt{F(X)}$ in a power series, we see that there is a unique monic
polynomial
\[
 P(X)=X^k+a_1X^{k-1}+\cdots+a_k\in\mathbb Q[X]
\]
such that
\[
 Q(X)=P(X)^2-F(X)
\]
has degree at most $k-1$.  
Observe that $t_i$ being distinct, $F(X)$ is a squarefree polynomial, which implies that $Q(X)$ is a non-zero polynomial.

Comparing the coefficient of
$X^{2k-j}$ gives
\[
 2a_j=c_j-\sum_{i=1}^{j-1}a_i a_{j-i}
 \qquad(1\leq j\leq k).
\]
Since the coefficients $c_j$ are integers, a straightforward induction shows that 
\begin{equation}
\label{eq: aj}
 a_j\ll  h^j \mand 2^{2j-1}a_j\in\mathbb Z, 
 \qquad(1\leq j\leq k).
\end{equation}

Thus, we may take $D_k=2^{2k-1}$ as a common denominator of the $a_j$'s;
in particular, $D_kP\in\mathbb Z[X]$.
Consequently,
\[
 |Q(x)|\ll  h^{2k}x^{k-1}.
\]

Suppose that $F(x)=y^2$ for some integer $y>0$. Since $c_j\ll h^j$, $1 \le j \le 2k$, 
 $x \gg h$  we have $F(x) = x^{2k} + O(h x^{2k-1})$. Similarly, from~\eqref{eq: aj} we have 
$P(x)=x^k  + O(h x^{k-1})$. Therefore, $P(x)+y\gg x^k$, provided that $h$ is large enough, and 
we see that
\[
 |P(x)-y| =\frac{|Q(x)|}{P(x)+y}  \ll h^{2k} x^{-1}.
\]
If 
\begin{equation}
\label{eq: small x}
x \ge c(r) h^{2k} =c(r)  h^r
\end{equation}
with some sufficiently large constant  $c(r) >0$, 
we conclude that $ |P(x)-y| < 1/D_k$.  Since
$D_k(P(x)-y)$ is an integer, we infer $P(x)=y$ and therefore $Q(x)=0$.

However, $D_k^2Q$ is
a nonzero polynomial with integer coefficients of size $O(h^{2k})$, and the elementary root bound gives
$x  \ll h^{2k}$.  Thus, increasing, if necessary, the value of the constant $c(r)$ in~\eqref{eq: small x}
we obtain a contradiction. 

This proves the claim for all even $r$.

Suppose now that $r$ is odd, and consider two distinct unordered
$r$-element subsets $\cS, \cT \subseteq \Z\cap(u,u+h]$ for which 
\[
\prod_{n\in \cS}n \mand \prod_{n\in \cT}n
\] are squares.
Then their symmetric difference 
$\cS\mathbin{\triangle}\cT$ is not empty and, 
\[
 \prod_{n\in \cS\mathbin{\triangle}\cT}n
 =
 \frac{\left(\prod_{n\in \cS}n\right)
       \left(\prod_{n\in \cT}n\right)}
      {\left(\prod_{n\in \cS\cap \cT}n\right)^2}
\]
is a perfect square.   Put $t=\# \(\cS\mathbin{\triangle}\cT\)$. Then $t$ is even
and $2\leq t\leq 2s$. Taking
\[
C(s)=\max_{\substack{2\leq j\leq 2s\\ j\ {\rm even}}} c(j),
\]
the hypothesis $u>C(s)h^{2s}$ implies $u>c(t)h^t$, contradicting
the above even-variable claim. 
\end{proof}

Consequently, in proving   Proposition~\ref{bound_nontrivial} we may now assume $u\leq h^A$, where $A>0$ is sufficiently
large and depends only on $s$.

\subsection{A squarefree-kernel estimate}
For a positive integer $n$, let $\sq(n)$ denote its largest square
divisor.

\begin{lemma}
If positive integers $n_1,\ldots,n_\ell$ satisfy
\[
n_1\cdots n_\ell=\square,
\]
then there is an upper triangular array of positive integers
\[
\mathbf b = 
\begin{pmatrix}
 b_{11}&b_{12}&\cdots&b_{1\ell}\\
       &b_{22}&\cdots&b_{2\ell}\\
       &      &\ddots&\vdots\\
       &      &      &b_{\ell\ell}
\end{pmatrix}
\]
such that, for $1\le r\le\ell$,
\[
b_{rr}^2=\sq(n_r),
\]
and
\[
n_r=(b_{1r}\cdots b_{rr})(b_{rr}\cdots b_{r\ell}).
\]
\end{lemma}
\begin{proof} This is a version of  a result of  Vaughan and Wooley~\cite{VW}, in the form used by Benatar, Nishry and 
Rodgers~\cite[Lemma~2.2]{BNR}.
\end{proof}

For convenience, we extend the notation symmetrically by setting
$b_{ji}=b_{ij}$ for $1\leq i<j\leq\ell$.  Therefore,
\begin{equation}
 \label{eq: brr dr}
n_r
 =b_{rr}^2 d_r, \quad \text{where} \quad d_r = \prod_{\substack{1\leq j\leq\ell\\j\ne r}}b_{rj}. 
\end{equation}
Since $b_{rr}^2=\sq(n_r)$, the integer $d_r$ is square-free.  Hence,
every $b_{rj}$ with $j\ne r$ is square-free.  Put
\[
\mathcal{T}_h
=\{b\in\N:b\leq h\ \text{and $b$ is square-free}\}.
\]
Let
\begin{equation}\label{def_Du}
 D_{u,h}(d)=\#\{a\in\N:~u<d  a^2\leq u+h\}.
\end{equation}
Since $b_{i,j} \mid n_i-n_j$, it is easy to see that
\begin{equation}
 \label{eq:tuple-pair-majorant}
 R_s^*(h,u)
 \leq
 \sum_{\substack{b_{rj}\in\mathcal{T}_h \\ r<j \leq s}}
 \prod_{r=1}^s D_{u,h}(d_r), 
\end{equation}
where $d_r$ is defined as in~\eqref{eq: brr dr}. 

We use version of the  Cauchy--Schwarz inequality given 
by Lemma~\ref{lem:pair-CS}  to bound the right-hand side of~\eqref{eq:tuple-pair-majorant}.

\subsection{Concluding the proof of Proposition~\ref{bound_nontrivial}}

By Lemma~\ref{lem:large-shift-distinct}, we may assume that
$u \ll h^{2s}$.  

For $s=1$, the distinctness condition is vacuous, and the interval
$(u,u+h]$ contains
$O(1+h/\sqrt{u})=O(h^{1/2})$ squares, so the result is immediate.
In the following context, we assume that $s\geq2$.

Recalling~\eqref{def_Du}, for $1\leq r\leq s$, put
\[
T_r =
\sum_{\substack{b_{rj}\in\mathcal{T}_h\\1\leq j\leq s,\ j\ne r}}
D_{u,h}^2(d_r).
\]
Applying Lemma~\ref{lem:pair-CS}, with $\cX_{ij}=\mathcal{T}_h$, to
\eqref{eq:tuple-pair-majorant}, we obtain
\begin{equation}
\label{eq:optimial_bdd}
R_s^*(h,u)
 \leq \prod_{r=1}^s T_r^{1/2}
 \ll
 \left(\max_{1\leq r\leq s}T_r\right)^{s/2}.
\end{equation}

It remains to show that, for every fixed $m\geq1$,
\[
\sum_{b_1,\ldots,b_m\in\mathcal{T}_h}
 D_{u,h}^2(b_1\cdots b_m) \ll h(\log h)^m.
\]
Write
\[
D_{u,h}(b_1\cdots b_m)
 =\sum_{u<b_1\cdots b_m a^2\leq u+h}1.
\]
Expanding the square, the left-hand side of the preceding estimate is
bounded by
\[
\sum_{b_1,\ldots,b_m\in\mathcal{T}_h}
\left(\sum_{u<b_1\cdots b_m a^2\leq u+h}1\right)^2
=
\sum_{b_1,\ldots,b_m\in\mathcal{T}_h}
\sum_{u<b_1\cdots b_m a^2\leq u+h}1+\Sigma_2, 
\]
where the term $\Sigma_2$ counts pairs
$b_1\cdots b_m a^2,b_1\cdots b_m c^2\in(u,u+h]$ with $a\ne c$.
Then $0<b_1\cdots b_m|a^2-c^2|<h$.  This implies that for
each fixed $(b_1,\ldots,b_m)$ the number of such pairs is
$O\(h/(b_1\cdots b_m)\)$. 
It follows that
\[
\Sigma_2
 \ll 
 h\sum_{b_1,\ldots,b_m\in\mathcal{T}_h}
 \frac{1}{b_1\cdots b_m}
 \leq
 h\left(\sum_{b\leq h}\frac1b\right)^m
 \ll h(\log h)^m.
\]
For the diagonal term, let $f_m(n)$ count the representations
\[
n=a^2b_1\cdots b_m,
\]
where $a,b_1,\ldots,b_m\in\N$ and every $b_i$ is square-free.
Let $\tau_k$ denote the $k$-fold divisor function, namely
\[
\tau_k(n)
=\#\{(d_1,\ldots,d_k)\in\N^k:~d_1\cdots d_k=n\}.
\]
Grouping $a^2b_1$ as a single factor gives an injective map
\[
(a,b_1,\ldots,b_m)
\longmapsto
(a^2b_1,b_2,\ldots,b_m)
\]
from the representations counted by $f_m(n)$ to the ordered
$m$-fold factorizations of $n$. Indeed, since $b_1$ is square-free,
$a$ and $b_1$ are uniquely determined by $a^2b_1$. Hence, 
\[
f_m(n)\leq\tau_m(n).
\]

Recalling that $u\ll h^{2s}$, the bound of Shiu~\cite[Theorem~1]{Shiu}
shows that the diagonal term is bounded by
\[
\sum_{u<n\leq u+h}\tau_m(n)
\ll h(\log h)^{m-1}. 
\]
Each $T_r$ involves exactly $s-1$
variables, and therefore~\eqref{eq:optimial_bdd} gives
\[
R_s^*(h,u)
 \ll \(h(\log h)^{s-1}\)^{s/2}
 =h^{s/2}(\log h)^{s(s-1)/2}.
\]
This proves  Proposition~\ref{bound_nontrivial} and thus 
Theorem~\ref{thm:short_bound}.

\section{Bounds for moments of character sums}

\subsection{Initial transformations}

Let
\[
M_{2s}^*(Q;\balpha,u,h)
 =\sum_{\substack{p\in[Q,2Q]\\p~\text{prime}}}
  |S_p^*(\balpha,u,h)|^{2s},
\]
where
\[
S_p^*(\balpha,u,h)
 =\sum_{\substack{u<n\le u+h\\n~\text{odd}}}\alpha_n\ls{n}{p}.
\]

Let $\ell$ be the integer determined by
\[
2^\ell\le h^{1/2}<2^{\ell+1}.
\]
Partitioning the integers $n\in(u,u+h]$ according to their $2$-adic
valuations $i\le\ell$, and bounding the contribution of the remaining
integers by $h/2^\ell+1\ll h^{1/2}$, we obtain
\[
S_p(\balpha,u,h)
 =\sum_{i=0}^\ell\ls{2^i}{p}
  S_p^*(\bbeta_i,u2^{-i},h2^{-i})+O(h^{1/2}),
\]
where
\[
\bbeta_i=(\beta_{i,n})_{u2^{-i}<n\le u2^{-i}+h2^{-i}},
\qquad \beta_{i,n}=\alpha_{2^in}.
\]

Writing $1=(i+1)^{-1}(i+1)$ and applying the H\"older inequality, we
obtain
\begin{align*}
|S_p(\balpha,u,h)|^{2s}
&\ll\left(\sum_{i=0}^\ell
 |S_p^*(\bbeta_i,u2^{-i},h2^{-i})|\right)^{2s}+h^s\\
&=\left(\sum_{i=0}^\ell(i+1)^{-1}(i+1)
 |S_p^*(\bbeta_i,u2^{-i},h2^{-i})|\right)^{2s}+h^s\\
&\le\left(\sum_{i=0}^\ell(i+1)^{-2s/(2s-1)}\right)^{2s-1}\\
&\qquad{}\times\sum_{i=0}^\ell(i+1)^{2s}
 |S_p^*(\bbeta_i,u2^{-i},h2^{-i})|^{2s}+h^s\\
&\ll\sum_{i=0}^\ell(i+1)^{2s}
 |S_p^*(\bbeta_i,u2^{-i},h2^{-i})|^{2s}+h^s.
\end{align*}
It follows that
\begin{equation}
\label{eq: M*-bound}
M_{2s}(Q;\balpha,u,h)
 \ll\sum_{i=0}^\ell(i+1)^{2s}
 M_{2s}^*(Q;\bbeta_i,u2^{-i},h2^{-i})+h^s\pi(Q).
\end{equation}

We next split $S_p^*(\balpha,u,h)$ according to the residue class
modulo $4$. Let $\cN_+$ and $\cN_-$ denote the sets of integers in
$(u,u+h]$ that are congruent to $1$ and $-1$ modulo $4$, respectively,
and, for $\sigma\in\{+,-\}$, define
\[
S_p^\sigma(\balpha,u,h)
 =\sum_{n\in\cN_\sigma}\alpha_n\ls{n}{p}
\]
and
\[
M_{2s}^\sigma(Q;\balpha,u,h)
 =\sum_{\substack{p\in[Q,2Q]\\p~\text{prime}}}
  |S_p^\sigma(\balpha,u,h)|^{2s}.
\]
Since
\[
M_{2s}^*(Q;\balpha,u,h)
 \ll M_{2s}^-(Q;\balpha,u,h)+M_{2s}^+(Q;\balpha,u,h),
\]
it is easy to see from~\eqref{eq: M*-bound} that to prove
Theorems~\ref{thm: Bound M2Quh} and~\ref{thm: Bound MsQuh}, it suffices
to establish
\begin{equation}
\label{eq: M2j-bound}
M_2^\sigma(Q;\balpha,u,h)
 \ll h\pi(Q)\log h+h^2Q^{7/8+o(1)}
\end{equation}
unconditionally, and
\begin{equation}
\label{eq: Msj-bound}
M_{2s}^\sigma(Q;\balpha,u,h)
 \ll h^s(\log h)^{s(2s-1)}\pi(Q)+h^{2s}Q^{1/2+o(1)}
\end{equation}
under the GRH, for each $\sigma\in\{+,-\}$.

\subsection{Using Selberg weights}

Fix $\varepsilon\in(0,1/3)$, set $z=Q^\varepsilon$, and let $\fP$ be
defined by~\eqref{eq: PrimeProd}. In the remainder of the proof, when
the denominator is not an odd prime, the symbol
$\left(\frac{\cdot}{q}\right)$ is understood as the Kronecker symbol.

For $\sigma\in\{+,-\}$, it follows from~\eqref{eq: sum lambda_e e|q}
that
\begin{align*}
M_{2s}^\sigma(Q;\balpha,u,h)
&=\sum_{\substack{p\in[Q,2Q]\\p~\text{prime}}}
 \left|\sum_{n\in\cN_\sigma}\alpha_n\left(\frac{n}{p}\right)\right|^{2s}\\
&\le\sum_{\substack{q\in[Q,2Q]\\\gcd(q,\fP)=1}}
 \left|\sum_{n\in\cN_\sigma}\alpha_n\left(\frac{n}{q}\right)\right|^{2s}\\
&\le\sum_{q\in[Q,2Q]}\sum_{e\mid q}\lambda_e^+
 \sum_{n_1,\ldots,n_{2s}\in\cN_\sigma}
 \alpha_{n_1}\cdots\alpha_{n_s}\\
&\qquad{}\times
 \overline{\alpha_{n_{s+1}}\cdots\alpha_{n_{2s}}}
 \left(\frac{n_1\cdots n_{2s}}{q}\right).
\end{align*}

The tuples for which $n_1\cdots n_{2s}$ is a square contribute at most
\[
|\fU_\square^{\sigma,s}|
 \le R_{2s}(h,u)\sum_{q\in[Q,2Q]}\sum_{e\mid q}\lambda_e^+.
\]
Hence, by~\eqref{eq: sum lambda_n dyadic},
\begin{equation}
\label{eq: square}
|\fU_\square^{\sigma,s}|\ll R_{2s}(h,u)\pi(Q).
\end{equation}

For $\sigma\in\{+,-\}$, we have
\[
n_1\cdots n_{2s}\equiv1\pmod4,\qquad
n_1,\ldots,n_{2s}\in\cN_\sigma.
\]
Quadratic reciprocity therefore shows that the contribution from tuples
for which $n_1\cdots n_{2s}$ is not a square is
\begin{equation}
\label{eq: non square prelim 1}
\begin{split}
\fU_{\boxtimes}^{\sigma,s}
=\sum_{q\in[Q,2Q]}&\sum_{e\mid q}\lambda_e^+
 \sum_{\substack{n_1,\ldots,n_{2s}\in\cN_\sigma\\
                  n_1\cdots n_{2s}\ne\square}}
 \alpha_{n_1}\cdots\alpha_{n_s}\\
&\hspace{35mm}\times
 \overline{\alpha_{n_{s+1}}\cdots\alpha_{n_{2s}}}
 \left(\frac{q}{n_1\cdots n_{2s}}\right).
\end{split}
\end{equation}

\subsection{The nonsquare contribution}
\label{sec: Bound U}

Changing the order of summation in~\eqref{eq: non square prelim 1} and
using~\eqref{eq:lambda_n = 0}, we obtain
\begin{equation}
\label{eq: non square prelim 2}
\begin{split}
\fU_{\boxtimes}^{\sigma,s}
=\sum_{\substack{n_1,\ldots,n_{2s}\in\cN_\sigma\\
                  n_1\cdots n_{2s}\ne\square}}
 &\alpha_{n_1}\cdots\alpha_{n_s}
 \overline{\alpha_{n_{s+1}}\cdots\alpha_{n_{2s}}}\\
&\quad\times\sum_{e\le z^2}\lambda_e^+
 \sum_{\substack{q\in[Q,2Q]\\e\mid q}}
 \ls{q}{n_1\cdots n_{2s}}.
\end{split}
\end{equation}

If $s=1$, Lemma~\ref{lem:Burg} gives
\begin{align*}
\left|\sum_{\substack{q\in[Q,2Q]\\e\mid q}}\ls{q}{n_1n_2}\right|
&=\left|\ls{e}{n_1n_2}
 \sum_{Q/e\le v\le2Q/e}\ls{v}{n_1n_2}\right|\\
&\ll(Q/e)^{1/2}(u+h)^{2\cdot3/16+o(1)}
 \le Q^{7/8+o(1)}e^{-1/2}.
\end{align*}
It follows from~\eqref{eq: non square prelim 2} that
\[
|\fU_{\boxtimes}^{\sigma,1}|
 \ll h^2Q^{7/8+o(1)}
 \sum_{e\le z^2}|\lambda_e^+|e^{-1/2}
 \le h^2Q^{7/8+o(1)}\sum_{e\le z^2}|\lambda_e^+|.
\]
Thus, by~\eqref{eq: sum|lambda_n|},
\begin{equation}
\label{eq: non-square 2}
|\fU_{\boxtimes}^{\sigma,1}|
 \le h^2Q^{7/8+o(1)}\frac{z^2}{(\log z)^2}.
\end{equation}

Under the GRH, Lemma~\ref{lem:GRH} gives, for every fixed $s\ge1$,
\begin{align*}
\left|\sum_{\substack{q\in[Q,2Q]\\e\mid q}}
 \ls{q}{n_1\cdots n_{2s}}\right|
&=\left|\ls{e}{n_1\cdots n_{2s}}
 \sum_{Q/e\le t\le2Q/e}\ls{t}{n_1\cdots n_{2s}}\right|\\
&\le (Q/e)^{1/2}(u+h)^{o(1)}
 \le Q^{1/2+o(1)}e^{-1/2}.
\end{align*}
Consequently,
\begin{equation}
\label{eq: non-square s}
|\fU_{\boxtimes}^{\sigma,s}|
 \ll h^{2s}Q^{1/2+o(1)}\frac{z^2}{(\log z)^2}.
\end{equation}

\subsection{Completion of the proof}

Combining~\eqref{eq: square} and~\eqref{eq: non-square 2}, and then
applying Theorem~\ref{thm:short_bound} with $s=2$, gives
\[
M_2^\sigma(Q;\balpha,u,h)
 \ll h\pi(Q)\log h
   +h^2Q^{7/8+o(1)}\frac{z^2}{(\log z)^2}.
\]
Given any $\delta>0$, choose $\varepsilon<\min\{1/3,\delta/4\}$. Since
$z=Q^\varepsilon$, the preceding estimate yields~\eqref{eq: M2j-bound}
with $Q^{7/8+\delta}$ in place of $Q^{7/8+o(1)}$ and the implied constant, 
which may depend on $\delta$. As $\delta$ is
arbitrary, this proves~\eqref{eq: M2j-bound}.

Similarly, under the GRH, combining~\eqref{eq: square}
and~\eqref{eq: non-square s} with Theorem~\ref{thm:short_bound}, 
applied to $R_{2s}(h,u)$,
gives
\[
M_{2s}^\sigma(Q;\balpha,u,h)
 \ll h^s(\log h)^{s(2s-1)}\pi(Q)
   +h^{2s}Q^{1/2+o(1)}\frac{z^2}{(\log z)^2}.
\]
Choosing $\varepsilon$ arbitrarily small proves~\eqref{eq: Msj-bound}
and implies Theorems~\ref{thm: Bound M2Quh} and~\ref{thm: Bound MsQuh}.

\section{Comments and open questions}
\label{sec:comm}
For each prescribed interval $(u,u+h]$, Theorem~\ref{thm: Bound M2Quh}
implies that, as $Q\to\infty$ and $h\to\infty$, quadratic residues and
non-residues are asymptotically equidistributed for almost all primes
$p\in[Q,2Q]$.
 
 
Although Theorems~\ref{thm: Bound M2Quh} and~\ref{thm: Bound MsQuh}
were stated in full generality, they are most useful when $h$ is very
small, so that the first term dominates. In this regime, the factor
$\pi(Q)$ in place of $Q$ is essential. Outside this range, one may extend
the summation from primes $p$ to all integers $q$ and use a simpler
argument that requires no sieve.

It is natural to ask whether an analogous result can be obtained for
primitive roots. Quadratic reciprocity plays a central role in our
argument, which therefore does not extend directly to the other characters
needed for this problem. For primitive roots in initial intervals, the
best currently known results and the method of~\cite{KST} make it possible
to study intervals $[1,h]$ of length
\[
h\ge\exp\left(\frac{A(\log\log Q)^2}{\log\log\log Q}\right)
\]
for some constant $A>0$. These methods do not apply to intervals away
from the origin.

Finally, a recent result of de la Bret{\`e}che, Wang, and
Xu~\cite{dlBWX} can be combined with our approach to study moments of
short character sums with polynomial arguments.

\section*{Acknowledgements} 

The authors are grateful to R{\'e}gis de la Bret{\`e}che  for very interesting discussions 
on bounding the quantity $\cR_s(h,u)$ and in particular for informing us about his work 
on an alternative approach to an optimal estimate of $\cR_s(h,u)$.

During the preparation of this work, M.~M. was supported by the French
National Research Agency (ANR) under project ANR-25-CE40-1961-01;
I.~S. was supported by Australian Research Council Grant DP230100534;
and Y.~X. was supported by the China Scholarship Council.


\bibliographystyle{plain}

\end{document}